\documentclass{amsart}
\usepackage{amscd,amsmath,amssymb,amsthm,bm,cases,color,comment,mathrsfs,mathtools}
\usepackage[dvips]{graphicx}
\usepackage[all]{xy}
\makeatletter
\@namedef{subjclassname@2020}{%
\textup{2020} Mathematics Subject Classification}
\makeatother
\theoremstyle{definition}
\newtheorem{definition}{Definition}[section]
\newtheorem{remark}[definition]{Remark}

\theoremstyle{plain}
\newtheorem{theorem}[definition]{Theorem}
\newtheorem{proposition}[definition]{Proposition}
\newtheorem{lemma}[definition]{Lemma}
\newtheorem{corollary}[definition]{Corollary}
\numberwithin{equation}{section}
\title[Hyperbolicity of nonseparating curve graphs]{Uniform hyperbolicity of nonseparating curve graphs of nonorientable surfaces}
\author[E.~Kuno]{Erika Kuno}
\address{
(Erika Kuno)
Department of Mathematics,
Graduate School of Science,
Osaka University,
1-1 Machikaneyama-cho, Toyonaka, Osaka 560-0043, Japan
}
\email{e-kuno@math.sci.osaka-u.ac.jp}
\date{\today}
\begin{document}
\begin{abstract}
Let $N$ be a connected finite type nonorientable surface with or without boundary components and punctures. We prove that the graph of nonseparating curves of $N$ is connected and Gromov hyperbolic with a constant which does not depend on the topological type of the surface by using the bicorn curves introduced by Przytycki and Sisto. The proof is based on the argument by Rasmussen on the uniform hyperbolicity of graphs of nonseparating curves for finite type orientable surfaces.
\end{abstract}
\maketitle

\section{Introduction}
\label{Introduction}
Let $N=N_{g,p}^{f}$ be a connected finite type nonorientable surface of genus $g$ with $f$ boundary components and $p$ punctures and $S=S_{g,p}^{f}$ be a connected finite type orientable surface of genus $g$ with $f$ boundary components and $p$ punctures. When $f=0$ or $p=0$, we drop the suffix that denotes $0$, excepting $g$, from $N_{g, p}^{f}$ and $S_{g, p}^{f}$. For example, $N_{g, 0}^{0}$ is simply denoted as $N_{g}$. If we are not interested in whether a given surface is orientable or not, we denote the surface by $F$.% A curve $c$ on $F$ is {\it essential} if it is nonseparating, or it is separating and does not bound a disk, a disk with one puncture, an annulus, or a M\"obius band on $F$. 

A simple closed curve $c$ on $F$ is said to be {\it one-sided} if the regular neighborhood of $c$ is a M\"obius band, and $c$ is said to be {\it two-sided} if the regular neighborhood of $c$ is an annulus. Moreover a simple closed curve $c$ on $F$ is {\it nonseparating} if the complement $F\backslash c$ of $c$ in $F$ is connected, and $c$ is {\it separating} if the complement $F\backslash c$ of $c$ in $F$ is not connected. We say a curve $c$ is {\it essential} if $c$ is nonseparating, or $c$ is separating and does not bound a disk, a disk with a puncture, an annulus, or a M\"obius band. In this paper, a curve on $F$ means a simple closed curve on $F$.

The {\it curve graph} $\mathcal{C}(F)$ of $F$ is the simplicial graph whose vertex set consists of the homotopy classes of all essential simple closed curves on $F$ and whose edge set consists of all non-ordered pairs of vertices which can be realized disjointly on $F$. We define the {\it nonseparating curve graph} $\mathcal{NC}(F)$ of $F$ as the full subgraph of $\mathcal{C}(F)$ consisting of all homotopy classes of nonseparating curves on $F$. We define the distances $d_{\mathcal{C}}(\cdot, \cdot)$ on $\mathcal{C}(F)$ and $d_{\mathcal{NC}}(\cdot, \cdot)$ on $\mathcal{NC}(F)$ by the minimal lengths of edge-paths connecting two vertices in $\mathcal{C}(F)$ and $\mathcal{NC}(F)$, respectively. Thus, we consider $\mathcal{C}(F)$ and $\mathcal{NC}(F)$ as geodesic spaces.

We denote by $i(a,b)$ the {\it geometric intersection number} between the homotopy classes of the curves $a$ and $b$ in $F$, that is, the minimal number of intersection points between a representative curve in the homotopy class of $a$ and a representative curve in the homotopy class of $b$. Two curves $a$ and $b$ in $F$ are in {\it minimal position} if the number of intersection points between $a$ and $b$ is minimal in the homotopy classes of $a$ and $b$. Note that two curves are in minimal position in $F$ if and only if they do not bound a bigon on $F$ (see~\cite[Proposition 2.1]{Stukow17} for nonorientable surfaces).  

For orientable surfaces, first Masur and Minsky~\cite{Masur--Minsky99} proved that the curve graphs of orientable surfaces are Gromov hyperbolic. %After their proof, various other proofs of hyperbolicity for curve graphs of orientable surfaces were given by several authors.
Aougab~\cite{Aougab13}, Bowditch~\cite{Bowditch14}, Clay, Rafi, and Schleimer~\cite{Clay--Rafi--Schleimer14}, and Hensel, Przytycki, and Webb~\cite{Hensel--Przytycki--Webb15} independently proved that the curve graphs of orientable surfaces are {\it uniformly hyperbolic}, that is, one can choose hyperbolicity constants which do not depend on the topological types of the surfaces. Moreover, Rasmussen~\cite{Rasmussen20} proved that the nonseparating curve graphs of orientable surfaces are also uniformly hyperbolic. We note that by using the result of Rasmussen, Bowden, Hensel, and Webb~\cite{Bowden--Hensel--Webb22} proved that for orientable surfaces ``fine curve graphs'' are Gromov hyperbolic and studied the bounded cohomology of orientable surface diffeomorphism groups.

%For nonorientable surfaces, Bestvina and Fujiwara~\cite{Bestvina--Fujiwara07} proved that $\mathcal{C}(N)$ is Gromov hyperbolic (Masur and Schleimer~\cite{Masur--Schleimer13} gave another proof), and the author~\cite{Kuno16} proved that the curve graphs $\mathcal{C}(N)$ are uniformly hyperbolic by using the unicorn paths argument introduced by Hensel, Przytycki, and Webb~\cite{Hensel--Przytycki--Webb15}. 
For nonorientable surfaces, Bestvina and Fujiwara~\cite{Bestvina--Fujiwara07} proved that the curve graphs of nonorientable surfaces are Gromov hyperbolic (Masur and Schleimer~\cite{Masur--Schleimer13} gave another proof), and the author~\cite{Kuno16} studied the uniform hyperbolicity of curve graphs for nonorientable surfaces by using the unicorn paths argument introduced by Hensel, Przytycki, and Webb~\cite{Hensel--Przytycki--Webb15}. 
%However, the hyperbolicity of the nonseparating curve graph $\mathcal{NC}(N)$ of a finite type nonorientable surface remains unknown. %Moreover, Atalan and Korkmaz~\cite{Atalan--Korkmaz14} proved that some curve graph consisting of the ``essential'' one-sided curves is connected for a nonorientable surface of genus $g\geq 4$; an essential one-sided curve is a curve whose complement in $N$ is a nonorientable surface. 

%and the result of Rasmussen implies that the nonseparating curve graph $\mathcal{NC}(S)$ for infinite type orientable surfaces $S$ with finite positive genus by the result of Aramayona and Valdez~\cite{Aramayona--Valdez18} (see \cite[Corollary 1.2]{Rasmussen20}).
% Aramayona and Valdez~\cite{Aramayona--Valdez18} showed that the result of Rasmussen implies that for orientable infinite type orientable surfaces $S$ with finite positive genus, the nonseparating curve graph $\mathcal{NC}(S)$ is uniformly hyperbolic.

%In this paper, we prove the nonorientable surface version of Rasmussen's result, that is, uniform hyperbolicity of the nonseparating curve graphs of finite type nonorientable surfaces:
In this paper, we generalize Rasmussen's result to the nonorientable surfaces, that is, we prove the uniform hyperbolicity of the nonseparating curve graphs of finite type nonorientable surfaces:

\begin{theorem}\label{unif_hyp_of_nonseparating_curve_graph}
There exists a constant $\delta>0$ such that for any finite type nonorientable surface $N$ of genus $g=1$ and $f+p\geq 3$, of genus $g=2$ and $f+p\geq 1$, or of genus $g\geq 3$, the nonseparating curve graph $\mathcal{NC}(N)$ of $N$ is connected, has infinite diameter, and $\delta$-hyperbolic.
\end{theorem}

We remark that for nonorientable surfaces of genus 1 and 2, the nonseparating curve graphs are not connected or consists of exactly one isolated vertex. Hence we consider ``augmented'' nonseparating curve graphs (and we call them also nonseparating curve graphs here) for the nonorientable surfaces of genus 1 and 2 (see Remark~\ref{disconnectedness_nonseparating_curve_graphs_for_low_genus}). The nonseparating curve graphs of  nonorientable surfaces of genus $g=1$ and $f+p\geq 3$ or of genus $g=2$ and $f+p\geq 1$ has infinite vertices.% Then we see that: 
%The nonseparating curve graphs of  nonorientable surfaces of genus 1 and the number of boundary components or punctures is at least 3, or  of genus 2 and the number of boundary components or punctures is at least 1 has infinite vertices. Then we see that: 
%We treat the nonorientable surfaces of genus 2 and the number of boundary components or punctures is 1 in Section~\ref{The_cases_of_genus_1 and_2}.
%\begin{theorem}\label{unif_hyp_of_nonseparating_curve_graph_low_genus}
%There exists a constant $\delta>0$ such that for a nonorientable surface $N$ of genus $g=1$ and $f+p\geq 3$ or of genus $g=2$ and $f+p\geq 1$, the nonseparating curve graph $\mathcal{NC}(N)$ of $N$ is connected, has infinite diameter, and $\delta$-hyperbolic.
%\end{theorem}
We treat the nonorientable surfaces of genus 1 and 2 in Section~\ref{The_cases_of_genus_1 and_2}.

%We also remark that for any finite type nonorientable surface $N$ of genus $g\geq 3$, the nonseparating curve graph $\mathcal{NC}(N)$ of $N$ has infinite diameter (Remark~\ref{infinite_diameter}):
%
%\begin{proposition}
%For any finite type nonorientable surface $N$ of genus $g\geq 3$, the nonseparating curve graph $\mathcal{NC}(N)$ of $N$ has infinite diameter. 
%\end{proposition}

We also note that as an analogy of Bowden, Hensel, and Webb~\cite{Bowden--Hensel--Webb22}, Theorem~\ref{unif_hyp_of_nonseparating_curve_graph} also has an application to study Gromov hyperbolicity for fine curve graphs of nonorientable surfaces and the bounded cohomology of nonorientable surface diffeomorphism groups (\cite{Kimura--Kuno21}).

%\begin{theorem}\label{unif_hyp_of_nonseparating_curve_graph_low_genus}
%There exists $\delta>0$ such that for any finite type nonorientable surface $N$ of genus $g=1, 2$, the nonseparating curve graph $\mathcal{NC}(N)$ of $N$ is connected and $\delta$-hyperbolic.
%\end{theorem}

We prove Theorem \ref{unif_hyp_of_nonseparating_curve_graph} by using the bicorn curves defined by Przytycki and Sisto~\cite{Przytycki--Sisto17} and applying the argument of Rasmussen~\cite{Rasmussen20} to nonorientable surfaces, but with some modifications.
Here we list some of the differences from the case of orientable surfaces.
%In the process of applying the argument by  Long, Margalit, Pham, Verberne, and Yao~\cite{Long--Margalit--Pham--Verberne--Yao} to nonorientable surfaces, even in cases where the argument for orientable surfaces as it is does not work for nonorientable surfaces, the results are obtained by making appropriate modifications.

\begin{itemize}
 \item Rasmussen~\cite{Rasmussen20} uses the signs of intersection points between two curves in the argument. Although the signs of intersection points between two curves cannot be defined on nonorientable surfaces, the argument by Rasmussen also makes sence by considering the signs of intersection points on each regular neighborhood of an arc including the intersection points.
 \item While the signs of intersection points between two curves are not enough to calculate the geometric intersection number of the two curves for nonorientable surfaces, additional properties, that is, one-sidedness and two-sidedness of curves allow us to deside the geometric intersection numbers of the two curves.
% \item For nonorientable surfaces, we cannnot calculate geometric intersection numbers between two curves by using only signs between the two curves. However we  
 \item In the proof of Lemma~\ref{connectedness_bicorn_graph}, we can not show the same claim as that of orientable surfaces since the signs of intersection points between two curves cannot decide the geometric intersection number between two curves (for the details, see Remark~\ref{difference_between_orientable_and_nonorientable_in_claim}).
\end{itemize}

We conclude this section by noting that by combining  the result of Rasmussen and the result of Aramayona and Valdez~\cite{Aramayona--Valdez18}, it follows that the nonseparating curve graphs $\mathcal{NC}(S)$ for infinite type orientable surfaces $S$ with finite positive genus are uniformly hyperbolic  (see \cite[Corollary 1.2]{Rasmussen20}).
The author does not know whether the nonseparating curve graphs $\mathcal{NC}(N)$ of infinite type nonorientable surfaces $N$ are uniformly hyperbolic. We hope that Theorem~\ref{unif_hyp_of_nonseparating_curve_graph} has some application to study infinite type nonorientable surfaces.
%To the best of the author's knowledge, there is no proof of the uniform hyperbolicity of nonseparating curve graphs $\mathcal{NC}(N)$ of infinite type nonorientable surfaces with finite positive genus because the nonorientable surface version of the result by Aramayona and Valdez~\cite{Aramayona--Valdez18} seems not to be clear.
%The author does not know whether the uniform hyperbolicity of nonseparating curve graphs $\mathcal{NC}(N)$ of infinite type nonorientable surfaces with finite positive genus because the nonorientable surface version of the result by Aramayona and Valdez~\cite{Aramayona--Valdez18} seems not to be clear.
%We conclude this section by noting that Theorem~\ref{unif_hyp_of_nonseparating_curve_graph} has an application to study the bounded cohomology of nonorientable surface diffeomorphism groups (\cite{Kimura--Kuno21}).

\section{Preliminaries}
\label{Preliminaries}

Let $N$ be a connected finite type nonorientable surface. At the biginning, $N$ may have boundary components and punctures. However, since the difference between bonudary components and punctures will not relavant for us, we will assume that $N$ has no punctures from now on.

In this section, we show the connectedness and the infinite diameterness of $\mathcal{NC}(N)$, i.e.\ the first and the second parts of Theorem~\ref{unif_hyp_of_nonseparating_curve_graph}. Then, we recall the criterion for hyperbolicity and the definition of bicorns, as well as describe the notation for the proof of Theorem~\ref{unif_hyp_of_nonseparating_curve_graph}.

Firstly, we prove that the nonseparating curve graphs are connected:
\begin{proposition}{\rm(cf.\ }\cite[Theorem 4.4]{Farb--Margalit12}{\rm)}\label{connectedness_of_nonsep_curve_graph} 
Let $N$ be a nonorientable surface of genus $g\geq 3$ with $f\geq 0$ boundary components. Then, the nonseparating curve graph $\mathcal{NC}(N)$ of $N$ is connected. 
\end{proposition}

\begin{proof}
The proof follows the argument by Farb and Margalit~\cite[Theorem 4.4]{Farb--Margalit12} for orientable surfaces.

Note that, from Ivanov~\cite{Ivanov87} (see also Szepietowski~\cite[Theorem 6.1]{Szepietowski08}), we know that the usual curve graph $\mathcal{C}(N)$ is connected. We will prove the proposition by induction on the number of boundary components $f$. If $f\leq 1$, let $a, b$ be any pair of vertices of $\mathcal{NC}(N)$. Since $\mathcal{C}(N)$ is connected, we can take a geodesic path $a=c_{0},c_{1},\cdots,c_{n-1},c_{n}=b$ in $\mathcal{C}(N)$. Suppose that $c_{i}$ is a separating curve in $N$, and $N\setminus c_{i}=F'\cup F''$. As we take a geodesic path, the vertices $c_{i-1}$ and $c_{i+1}$ must intersect, and they are included in the same connected component $F'$. Both $F'$ and $F''$ have positive genus because $f\leq 1$. Hence, we can take a nonseparating curve $d$ in $F''$ and replace the separating curve $c_{i}$ with it. By repeating this argument for every separating curve in the geodesic, we obtain a path in $\mathcal{NC}(N)$ that connects $a$ and $b$.

For the induction on $f$, we assume that $f\geq 2$ and proceed as above. The only problem is that a separating curve $c_{i}$ may bound a subsurface $F''$ of genus $0$. In this case, the other subsurface $F'$ containing $c_{i-1}$ and $c_{i+1}$ is a nonorientable surface of genus $g$ that has at most $f-2$ boundary components. By induction, we can find a path connecting $c_{i-1}$ and $c_{i+1}$, and we replace $c_{i}$ by the corresponding sequence.
\end{proof}

\begin{remark}\label{disconnectedness_nonseparating_curve_graphs_for_low_genus}
For $g=1$ and $f\leq 1$, the nonseparating curve graphs consist of an isolated vertex, for $g=1$ and $f=2$, the nonseparating curve graphs consist of two isolated vertices, and for $g=1$ and $f\geq 3$, the nonseparating curve graphs consist of an infinite number of isolated vertices. For $g=2$ and $f\geq 0$, the nonseparating curve graphs are not connected. In fact, the complement of any nonseparating two-sided curve in a nonorientable surface of genus 2 with any number of boundary components is always orientable (and so we cannot take any one-sided curve in the complement). (See also \cite[Section 2.4]{Atalan--Korkmaz14} for the curve graphs in low dimensions for nonorientable surfaces.)
Therefore, we see that a nonseparating curve graph $\mathcal{NC}(N)$ which has at least two vertices is connected if and only if the genus of $N$ is at least three.
We will discuss the hyperbolicity of the ``augmented'' nonseparating curve graphs for a nonorientable surface of genus $2$ with one boundary component in Section~\ref{The_cases_of_genus_1 and_2}.
\end{remark}

Next, we prove that the nonseparating curve graphs have infinite diameters:
\begin{proposition}\label{infinite_diameter}
For any finite type nonorientable surface $N$ of genus $g\geq 3$, the nonseparating curve graph $\mathcal{NC}(N)$ of $N$ has infinite diameter. 
\end{proposition}

\begin{proof}
%We see that the connected nonseparating curve graphs for nonorientable surfaces have infinite diameters. 
Let $N=N_{g}^{f}$ be a nonorientable surface and $S$ the orientation double cover of $N$. %(hence $S=S_{g-1}^{2b}$).
Let $c$ be a vertex of the curve graph $\mathcal{C}(N)$ of $N$. We take a pseudo-Anosov element $\varphi$ of the mapping class group $\mathrm{Mod}(N)$ of $N$, and fix any $n\in\mathbb{Z}$. We write $\mathcal{G}$ as a geodesic which connects $c$ and $\varphi^{n}(c)$ in $\mathcal{C}(N)$. We denote the consecutive vertices in $\mathcal{G}$ by $\delta_{0}, \delta_{1}, \cdots, \delta_{m}$, where $\delta_{0}=c$ and $\delta_{m}=\varphi^{n}(c)$. 
%We put $\delta_{0}=c$, and denote the consective vertices in $\mathcal{G}$ by $\delta_{0}, \delta_{1}, \cdots, \delta_{m}$ (hence $\delta_{m}=\varphi^{n}(c)$). 
Since $\delta_{i}$ and $\delta_{i+1}$ can be realized disjointly on $N$ for any $i=0,\cdots, m-1$, any pair of lifts of $\delta_{i}$ and $\delta_{i+1}$ as vertices of $\mathcal{C}(S)$ can also be realized disjointly on $S$. 
%Since $\delta_{i}$ and $\delta_{i+1}$ can be realized disjointly on $N$ for any $i=0,\cdots, m-1$, any pair of vertices of $\mathcal{C}(S)$ in the lifts of $\delta_{i}$ and $\delta_{i+1}$ can also be realized disjointly on $S$. 
%We note that for a vertex $d$ of $\mathcal{C}(N)$ represented by a one-sided curve we take the union of the two arcs in the lift of $d$, which is a vertex of $\mathcal{C}(S)$. 
We take a lift $\gamma\in\mathcal{C}(S)$ of $c=\delta_{0}$. Let $\gamma^{1},\cdots, \gamma^{m-1}$ be lifts of $\delta_{1},\cdots,\delta_{m-1}$, respectively. Moreover, we choose the lift $\tilde{\varphi}^{n}(\gamma)\in\mathcal{C}(S)$ of $\varphi^{n}(c)$, where $\tilde{\varphi}\in\mathrm{Mod}(S)$ is a lift of $\varphi$ which is orientation-preserving, and put $\gamma^{m}=\tilde{\varphi}^{n}(\gamma)$.
Then, $\{\gamma^{i} \}_{i=0}^{m}$ is a path in $\mathcal{C}(S)$ which connects $\gamma$ and $\tilde{\varphi}^{n}(\gamma)$, and the length is the same as that of $\mathcal{G}$. Therefore, we have $d_{\mathcal{C}(S)}(\gamma,\tilde{\varphi}^{n}(\gamma))\leq d_{\mathcal{C}(N)}(c,\varphi^{n}(c))$.
By Masur and Minsky~\cite{Masur--Minsky99} a pseudo-Anosov element $\tilde{\varphi}\in\mathrm{Mod}(S)$ acts on $\mathcal{C}(S)$ loxodromically (see also Przytycki and Sisto~\cite{Przytycki--Sisto17}), and so it follows that there exists a constant $\mathcal{E}>0$ such that $\mathcal{E}|n|\leq d_{\mathcal{C}(S)}(\gamma,\tilde{\varphi}^{n}(\gamma))$. Thus we have $\mathcal{E}|n|\leq d_{\mathcal{C}(N)}(c,\varphi^{n}(c))$ (and we see that any pseudo-Anosov elment of $\mathrm{Mod}(N)$ acts on $\mathcal{C}(N)$ loxodromically). Since the nonseparating curve graph $\mathcal{NC}(N)$ is a full subgraph of the curve graph $\mathcal{C}(N)$, for any $c\in\mathcal{NC}(N)$, there exists $\mathcal{E}>0$ such that
\begin{equation*}
\mathcal{E}|n|\leq d_{\mathcal{C}(N)}(c,\varphi^{n}(c))\leq d_{\mathcal{NC}(N)}(c,\varphi^{n}(c)),
\end{equation*}
and we see that the nonseparating curve graph $\mathcal{NC}(N)$ has infinite diameter.
%For any vertex $c$ which is represented by a nonseparating curve on $N$ we also have the same inequality (and we see any pseudo-Anosov elment of $\mathrm{Mod}(N)$ also acts on $\mathcal{NC}(S)$ loxodromically), and we see that the nonseparating curve graph $\mathcal{NC}(N)$ has a infinite diameter.

%Let $\mathcal{C}_{\mathrm{two}}^{\pm}(N)$ be 
%Let $c$ be any homotopy class of an essential simple closed curve on $N$ which is two-sided. We take a pseudo-Anosov element $\phi$ of the mapping class group $\mathrm{Mod}(N)$ of $N$. By Masur--Schleimer~\cite[Lemma 6.3]{Masur--Schleimer13}, there is a quasi-isometric embedding $\nu\colon\mathcal{C}_{\mathrm{two}}^{\pm}(N)\rightarrow\mathcal{C}(S)$. Hence there exists a constant $\lambda\geq 1$ such that
%\begin{equation*}
%\frac{1}{\lambda}d_{\mathcal{C}{\mathrm{two}}^{\pm}(N)}(c,\phi^{n}(c))-\lambda\leq d_{\mathcal{C}(S)}(c,\phi^{n}(c))\leq\lambda d_{\mathcal{C}{\mathrm{two}}^{\pm}(N)}(c,\phi^{n}(c))+\lambda.
%\end{equation*}
%We take the lift $\tilde{\varphi}\in\mathrm{Mod}(S)$ such that
\end{proof}

We say that a geodesic space is $\delta$-{\it hyperbolic} if, in every geodesic triangle, each side lies in a $\delta$-neighborhood of the union of the other two. The following criterion for hyperbolicity is used in the argument of Rasmussen~\cite{Rasmussen20} to prove uniform hyperbolicity of nonseparating curve graphs $\mathcal{NC}(S)$ for finite type orientable surfaces. We also use it here to prove the uniform hyperbolicity of nonseparating curve graphs $\mathcal{NC}(N)$ for finite type nonorientable surfaces $N$.

\begin{proposition}{\rm(}\cite[Proposition 3.1]{Bowditch14},~\cite[Theorem 3.15]{Masur--Schleimer13}{\rm)}\label{criterion_for_hyperbolicity} 
Let $X$ be a graph and $a$ and $b$ two distinct vertices of $X$. Suppose that $A(a,b)$ is a connected subgraph of $X$ containing $a$ and $b$. Suppose also that there exists a constant $D>0$ such that
\begin{itemize}
\item[(i)] if $a$ and $b$ are joined by an edge, then the diameter of $A(a,b)$ is at most $D$,
\item[(i\hspace{-1pt}i)] for any triple $a$, $b$, $c$ of the vertices of $X$, $A(a,c)$ is contained in the $D$-neighborhood of $A(a,b)\cup A(b,c)$.
\end{itemize} 
Then, $X$ is hyperbolic with a hyperbolicity constant depending only on $D$.
\end{proposition}

In this paper, for a curve $a$ on $N$, we denote by $[a]$ the homology class represented by $a$ in $H_{1}(N,\partial N; \mathbb{Z}_{2})$; the $\mathbb{Z}_{2}$-coefficient relative homology group of $N$.
Useful facts are that
\begin{itemize}
\item a simple closed curve $a$ is separating if and only if $[a]=0$, and
\item a simple closed curve $a$ is nonseparating if and only if there exists a curve $b$ such that $i(a,b)=1$. We remark that one-sided curves are always nonseparating.
\end{itemize}

For an oriented curve or an oriented arc $a$, we write $\overline{xy}_{a}$ for a subarc of $a$ from $x$ to $y$ according to the orientation of $a$. If we reverse the orientation of the arc $\overline{xy}_{a}$, we denote by $-\overline{xy}_{a}$ the reversed oriented arc.

\begin{definition}
Let $a$ and $b$ be two curves (including arcs) on $N$ which intersect at least two times. Choose $x,y\in a\cap b$. Let $\alpha$ and $\beta$ be subarcs of $a$ and $b$, respectively whose endpoints are $x$ and $y$. If $\alpha$ and $\beta$ intersect at their endpoints $x$ and $y$ and nowhere in their interiors, we say that $\alpha\cup\beta$ is a {\it bicorn curve} or simply a {\it bicorn} between $a$ and $b$, and we call $\alpha$ and $\beta$, respectively the $a$-arc and the $b$-arc of the bicorn, and $x$ and $y$ the {\it corners} of the bicorn. We also consider $a$ and $b$ themselves to be bicorns between $a$ and $b$.
%If $\alpha\cup\beta$ is an embedded curve on $N$, we say that $\alpha\cup\beta$ is a {\it bicorn curve} or simply a {\it bicorn} between $a$ and $b$, and we call $\alpha$ and $\beta$, respectively the $a$-arc and the $b$-arc of the bicorn, and $x$ and $y$ the {\it corners} of the bicorn. We also consider $a$ and $b$ themselves to be bicorns between $a$ and $b$.
\end{definition}

%For vertices $a$ and $b$ in $\mathcal{NC}(N)$, we define $A(a,b)$ to be the full subgraph of $\mathcal{NC}(N)$ spanned by all of the nonseparating bicorns between $a$ and $b$.

In this paper, for a subset $M$ of a nonorientable surface $N$, we denote by $\mathrm{nbd}(M)$ a regular neigborhood of $M$ in $N$.

\section{Nonseparating curve graphs of nonorientable surfaces are uniformly hyperbolic}
\label{Nonseparating curve graphs of nonorientable surfaces are uniformly hyperbolic}

Let $N$ be any finite type nonorientable surface whose genus is at least $3$. We recall that since the difference between boundary components and punctures is no longer relavant for us, we assume that nonorientable surfaces have no punctures. Unless it causes confusion, abusing the notation, we realize vertices on nonseparating curve graphs as nonseparating curves which are mutually in minimal position from now on.
%Unless it causes confusion, abusing the notation, we also seem vertices on nonseparating curve graphs as the representatives which are mutually in minimal position from now on.

The goal of this section is to prove the second half of Theorem~\ref{unif_hyp_of_nonseparating_curve_graph}, that is, the uniform hyperbolicity of nonseparating curve graphs for nonorientable surfaces. From this section, we denote by $\mathcal{NC}(N)$ the graph of nonseparating curves defined as follows: the vertices are the homotopy classes of nonseparating curves, and an edge joins two vertices if we can choose representatives of vertices which intersect at most twice. Moreover, we denote by $\mathcal{NC}'(N)$ the nonseparating curve graph defined in the usual way, that is, the vertices are the homotopy classes of nonseparating curves, and the edges correspond to the pairs of vertices which can be realized disjointly. Let $d_{\mathcal{NC}}(\cdot,\cdot)$ and $d_{\mathcal{NC}'}(\cdot,\cdot)$ be the metrics on $\mathcal{NC(N)}$ and $\mathcal{NC}'(N)$, respectively. We can show that $\mathcal{NC(N)}$ and $\mathcal{NC}'(N)$ are quasi-isometric:

\begin{lemma}{\rm(cf. }\cite[Lemma 3.1]{Rasmussen20}{\rm)}\label{nonseparating_curve_graphs_are_quasi-isometric}
Let $a$ and $b$ be any pair of vertices of $\mathcal{NC}'(N)$. Then, we have $d_{\mathcal{NC}'}(a,b)\leq 2i(a,b)+1$.
\end{lemma}

\begin{proof}[Proof of Lemma~\ref{nonseparating_curve_graphs_are_quasi-isometric}]
%Suppose that $a$ and $b$ are nonseparating curves which are in minimal position. 
We will prove the lemma by induction on the intersection number $i(a,b)$ between $a$ and $b$. 
%Slightly abusing the notation we realize them as nonseparating curves which are in minimal position. 
When $i(a,b)=1$, the regular neighborhood $\mathrm{nbd}(a\cup b)$ of $a\cup b$ is any one of $N_{1,2}$, $N_{2,1}$, or $S_{1,1}$. Since the genus of $N$ is at least $3$, at least one component of $N-\mathrm{nbd}(a\cup b)$ has positive genus. We can take a nonseparating curve $c$ on $N-\mathrm{nbd}(a\cup b)$, and $c$ is disjoint from both $a$ and $b$. Hence, we have $d_{\mathcal{NC}'}(a,b)\leq d_{\mathcal{NC}'}(a,c)+d_{\mathcal{NC}'}(c,b)=2\leq 2i(a,b)+1$, and we are done.

When $i(a,b)\geq 2$, we choose an orientation of $b$. We take an intersection point $x\in a\cap b$. Let $y$ be the first point of $a\cap b$ after $x$ along $b$ according to the orientation of $b$. We put $\beta=\overline{xy}_{b}$, and so $\beta$ does not intersect $a$ on its interior. We fix a regular neighborhood $\mathrm{nbd}(\beta)$ of $\beta$. We note that $\mathrm{nbd}(\beta)$ is homeomorphic to a disk, and so it is orientable. We have two cases where the signs in $\mathrm{nbd}(\beta)$ at the intersection points $x$ and $y$ between $a$ and $b$ are the same and different (see Figure~\ref{fig_consective_intersections_between_a_and_b_v2}). We define two curves $c_{1}$ and $c_{2}$ as shown in Figure~\ref{fig_how_to_make_c_1_and_c_2_v3}.

\begin{figure}[ht]
\includegraphics[scale=0.35]{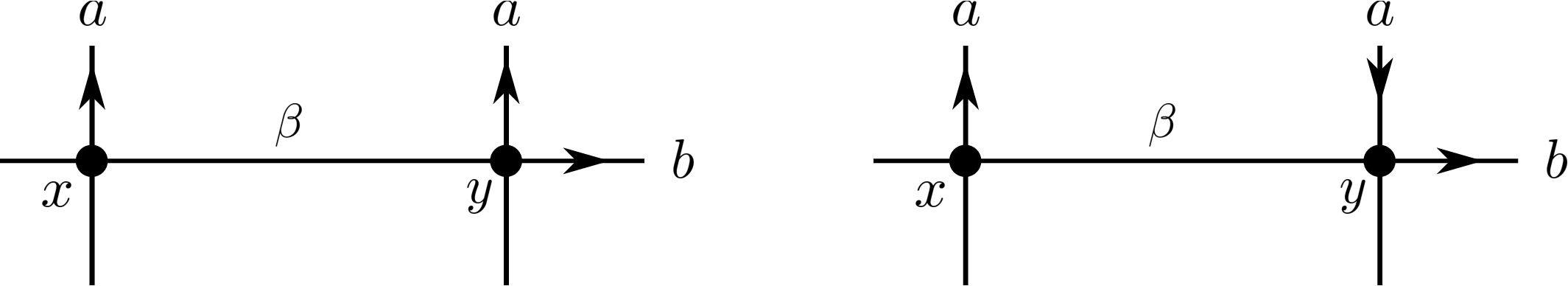}
\caption{Two points $x$ and $y$ of $a\cap b$ consective along $b$ whose signs in  $\mathrm{nbd}(\beta)$ are the same (left-hand side) and different (right-hand side).}\label{fig_consective_intersections_between_a_and_b_v2}
\end{figure}

\begin{figure}[ht]
\includegraphics[scale=0.35]{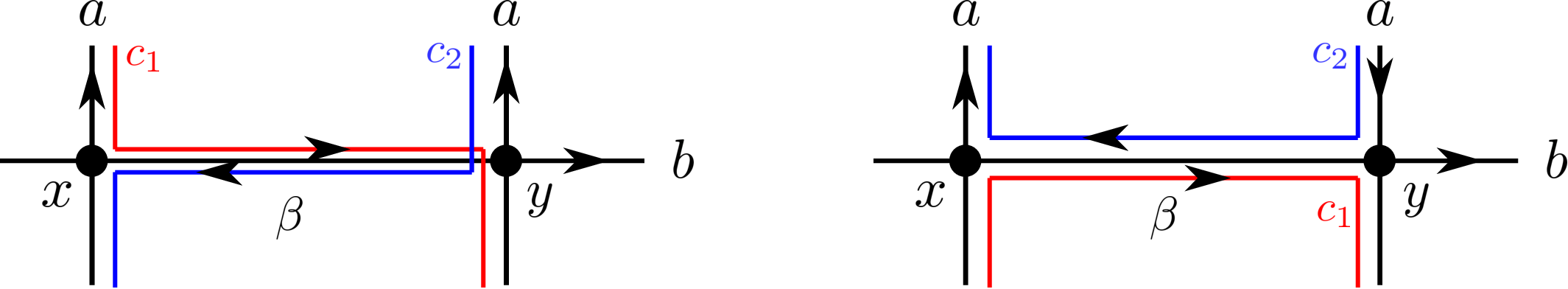}
\caption{The way to make new curves $c_{1}$ and $c_{2}$ when the signs in  $\mathrm{nbd}(\beta)$ at the intersection points $x$ and $y$ between $a$ and $b$ are the same (left) and different (right).}\label{fig_how_to_make_c_1_and_c_2_v3}
\end{figure}

First, consider the case where the signs in  $\mathrm{nbd}(\beta)$ at the intersection points $x$ and $y$ between $a$ and $b$ are the same. It follows that $[a]=[c_{1}]+[c_{2}]$ in $H_{1}(N,\partial N; \mathbb{Z}_{2})$ (again, note that we are considering the $\mathbb{Z}_{2}$-coefficient relative homology group of $N$). Since $a$ is nonseparating, we have $[a]\not =0$, so at least one of $[c_{1}]$ and $[c_{2}]$ is not $0$. Without loss of generality we may assume that $[c_{1}]\not =0$, and we put $c=c_{1}$. Accordingly, we see that $i(a,c)=1$ if $c$ is two-sided and $i(a,c)=0$ if $c$ is one-sided. Moreover, we have $i(b,c)\leq i(a,b)-2$ in both cases where $c$ is two-sided and where $c$ is one-sided. By induction, we have
\begin{equation*}
d_{\mathcal{NC}'}(a,b)\leq d_{\mathcal{NC}'}(a,c)+d_{\mathcal{NC}'}(c,b)\leq 2+2i(c,b)+1\leq 2+2(i(a,b)-1)+1\leq 2i(a,b)+1.
\end{equation*}

Second, consider the case where the signs in  $\mathrm{nbd}(\beta)$ at the intersections $x$ and $y$ between $a$ and $b$ are different. Here, it also follows that $[a]=[c_{1}]+[c_{2}]$. Similarly, we may assume that $[c_{1}]\not =0$ and put $c=c_{1}$. Then, we see that $i(a,c)=0$ if $c$ is two-sided and $i(a,c)=1$ if $c$ is one-sided. Moreover, we have $i(b,c)\leq i(a,b)-2$, particularly $i(b,c)\leq i(a,b)-1$ in both cases where $c$ is two-sided and where $c$ is one-sided. By induction, we have
\begin{equation*}
d_{\mathcal{NC}'}(a,b)\leq d_{\mathcal{NC}'}(a,c)+d_{\mathcal{NC}'}(c,b)\leq 2+2i(c,b)+1\leq 2+2(i(a,b)-1)+1\leq 2i(a,b)+1.
\end{equation*}

%\begin{align*}
%  \begin{aligned}
%  d_{\mathcal{NC}'}(a,b) & \leq
%    d_{\mathcal{NC}'}(a,c)+d_{\mathcal{NC}'}(c,b) \\
%   & \leq  2+2i(c,b)+1 \\
%   & \leq 2+2(i(a,b)-1)+1 \\
%   & \leq 2+2(i(a,b)-1)+1 \\
%   & \leq 2i(a,b)+1.
%\end{aligned}
%\end{align*}

\end{proof}

\begin{corollary}\label{cor_nonseparating_curve_graphs_are_quasi-isometric}
Let $N$ be a finite type nonorientable surface whose genus is at least $3$. Then, $\mathcal{NC}(N)$ and $\mathcal{NC}'(N)$ are quasi-isometric.
\end{corollary}

\begin{proof}[Proof of Corollary~\ref{cor_nonseparating_curve_graphs_are_quasi-isometric}]
We show the natural inclusion $\iota\colon\mathcal{NC}'(N)\rightarrow\mathcal{NC}(N)$, $\iota(a)=a$ is a quasi-isometry. Let $a$ and $b$ be any vertices in $\mathcal{NC}'(N)$.
Since $\mathcal{NC}'(N)\subset\mathcal{NC}(N)$, it follows that $d_{\mathcal{NC}}(\iota(a),\iota(b))\leq d_{\mathcal{NC'}}(a,b)$. We consider the opposite direction of the ineqality.
We assume that $d_{\mathcal{NC}}(\iota(a),\iota(b))=n$. Then there exsits a sequance $a=c_{0}, c_{1},\cdots,c_{n-1},c_{n}=b$ such that $i(c_{i},c_{i+1})\leq 2$ for any $i=0,1,\cdots,n-1$. By Lemma~\ref{nonseparating_curve_graphs_are_quasi-isometric}, we have
\begin{equation*}
d_{\mathcal{NC}'}(c_{i},c_{i+1})\leq 2i(c_{i},c_{i+1})+1\leq 5=5d_{\mathcal{NC}}(c_{i},c_{i+1}).
\end{equation*}
Thus we obtain
\begin{equation*}
d_{\mathcal{NC}'}(a,b)\leq 5d_{\mathcal{NC}}(\iota(a),\iota(b)).
\end{equation*}
Therefore we see that $\iota$ is a quasi-isometric embedding.

For any vertex $a\in\mathcal{NC}(N)$, we choose $a\in\mathcal{NC}'(N)$. Then we have 
\begin{equation*}
d_{\mathcal{NC}}(a,\iota(a))=d_{\mathcal{NC}}(a,a)=0.
\end{equation*}
Therefore we see that  $\iota$ is quasi-dense.
\end{proof}

For vertices $a$ and $b$ in $\mathcal{NC}(N)$, we define $A(a,b)$ to be the full subgraph of $\mathcal{NC}(N)$ spanned by the homotopy classes of the nonseparating bicorns between $a$ and $b$.

From Corollary~\ref{cor_nonseparating_curve_graphs_are_quasi-isometric} it is enough to prove the uniform hyperbolicity of $\mathcal{NC}(N)$ for Theorem~\ref{unif_hyp_of_nonseparating_curve_graph}. Through the following steps, we will check that $\mathcal{NC}(N)$ and the full subgraphs $A(a,b)$ of $\mathcal{NC}(N)$ satisfy the conditions of Proposition~\ref{criterion_for_hyperbolicity}:
\begin{itemize}
\item find a uniform bound on the diameter of $A(a,b)$ whose vertices $a$ and $b$ are connected by an edge (Lemma~\ref{bounded_diameter}), 
\item prove the connectedness of $A(a,b)$ (Lemma~\ref{connectedness_bicorn_graph}),
\item show the slim triangle properties (Lemma~\ref{slim_triangle_property}).
\end{itemize}

First, we prove Lemma~\ref{bounded_diameter}.
\begin{lemma}{\rm(cf. }\cite[Proposition 3.2]{Rasmussen20}{\rm)}\label{bounded_diameter} 
Let $a$ and $b$ be any pair of vertices of $\mathcal{NC}(N)$ that are connected by an edge. Then, $A(a,b)$ has a diameter at most two. 
\end{lemma}

\begin{proof}[Proof of Lemma~\ref{bounded_diameter}]
If $i(a,b)\leq 1$, then $A(a,b)$ consists of two vertices $a$, $b$, and exactly one edge connects $a$ and $b$, so we are done.

If $i(a,b)=2$, we fix an orientation of $b$. We show that any vertex $c\not =b$ of $A(a,b)$ satisfies $i(a,c)\leq 1$. Take any bicorn $c\not= b$ between $a$ and $b$, and put $c=\alpha\cup\beta$, where $\alpha$ is the $a$-arc and $\beta$ is the $b$-arc of $c$. According to the orientation of $b$, we assume that $\beta$ starts at $x\in a\cap b$ and ends at $y\in a\cap b$ (We note that now $a$ and $b$ satisfies $i(a,b)=2$, and so it follows that $a\cap b=\{x,y\}$). 
%Put $a\cap b=\{x,y\}$, and let $c=\alpha\cup\beta$ be a bicorn between $a$ and $b$ whose corners are $x$ and $y$.

First, consider the case where the signs in $\mathrm{nbd}(\beta)$ at $x$ and $y$ between $a$ and $b$ are the same. We see that $i(a,c)=1$ if $c$ is two-sided, and $i(a,c)=0$ if $c$ is one-sided (see Figure~\ref{fig_intersection_number_between_bicorns_for_same_signs_v1}). Second, consider the case where the signs in $\mathrm{nbd}(\beta)$ at $x$ and $y$ between $a$ and $b$ are different. We see that $i(a,c)=0$ if $c$ is two-sided, and $i(a,c)=1$ if $c$ is one-sided (see Figure~\ref{fig_intersection_number_between_bicorns_for_opposite_signs_v1}), and we are done.

\begin{figure}[ht]
\includegraphics[scale=0.35]{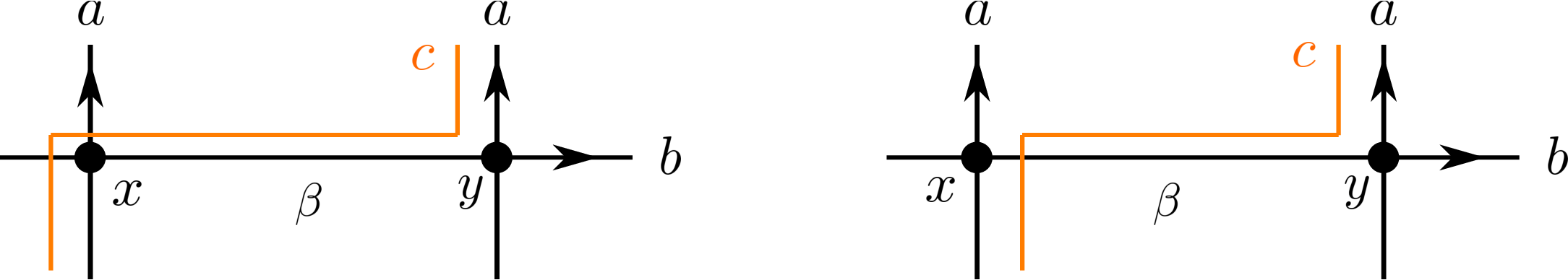}
\caption{Examples of intersections between $a$ and $c$ if $c$ is two-sided (left) and one-sided (right).}\label{fig_intersection_number_between_bicorns_for_same_signs_v1}
\end{figure}

\begin{figure}[ht]
\includegraphics[scale=0.35]{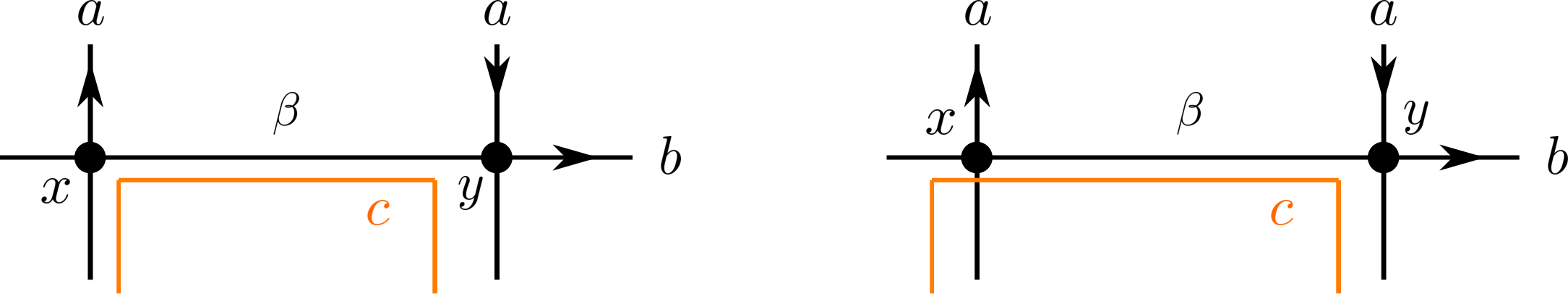}
\caption{Examples of intersections between $a$ and $c$ if $c$ is two-sided (left) and one-sided (right).}\label{fig_intersection_number_between_bicorns_for_opposite_signs_v1}
\end{figure}

\end{proof}

\begin{lemma}{\rm(cf. }\cite[Claim 3.4]{Rasmussen20}{\rm)}\label{connectedness_bicorn_graph} 
For any pair $a$ and $b$ of vertices in $\mathcal{NC}(N)$, the graph $A(a,b)$ is connected.
\end{lemma}

\begin{proof}[Proof of Lemma~\ref{connectedness_bicorn_graph}]
If $i(a,b)\leq 1$, then $A(a,b)$ consists of exactly two vertices $a$ and $b$ and a single edge connects $a$ and $b$, and hence $A(a,b)$ is connected. So, we assume that $i(a,b)\geq 2$.

We define a partial order on the vertices of $A(a,b)$ as follows: $c<c'$ if the $b$-arc of $c'$ properly contains the $b$-arc of $c$. It is enough to prove that given any $c\in A(a,b)$, if $c\not=b$, we can find $c'\in A(a,b)$ so that $c<c'$ and $i(c,c')\leq 2$.

We fix an orientation of $b$. First, consider the case where $c=a$. Choose an intersection point $x\in a\cap b$. Let $y$ be the first point of $a\cap b$ after $x$ along $b$ according to the orientation of $b$. Put $\beta=\overline{xy}_{b}\subset b$ so that the interior of $\beta$ is disjoint from $a$. The points $x$ and $y$ bound the two subarcs $\alpha$ and $\alpha'$ of $a$ with disjoint interiors such that $a=\alpha\cup\alpha'$. Accordingly, $c_{1}'=\alpha\cup\beta$ and $c_{2}'=\alpha'\cup\beta$ are bicorns between $a$ and $b$. For any $i=1,2$, we see that $i(c_{i}', c)\leq 1$. In Fact, first consider the case where the signs in  $\mathrm{nbd}(\beta)$ at the intersection points $x$ and $y$ between $a$ and $b$ are the same. Then, we see that $i(c_{i}', c)=1$ if $c_{i}'$ is two-sided, and $i(c_{i}', c)=0$ if $c_{i}'$ is one-sided. Second consider the case where the signs in  $\mathrm{nbd}(\beta)$ at the intersection points $x$ and $y$ between $a$ and $b$ are different. we see that $i(c_{i}', c)=0$ if $c_{i}'$ is two-sided, and $i(c_{i}', c)=1$ if $c_{i}'$ is one-sided. Moreover, we have $[a]=[c_{1}']+[c_{2}']$. Since $a$ is nonseparating, at least one of $c_{1}$ and $c_{2}$ is nonseparating. Without loss of generality we may assume that $[c_{1}']\not=0$ and put $c'=c_{1}'\in A(a,b)$. Then, it follows that $c<c'$, and $c$ and $c'$ are joined by an edge in $A(a,b)$.

Second, consider the case where $c\not=a$. We put $c=\alpha\cup\beta$, where $\alpha$ is a subarc of $a$ and $\beta$ is a subarc of $b$, and we assume that $\beta$ starts at $x$ and ends at $y$ according to the orientation of $b$. In this case, we extend $\beta$ past both of its endpoints until it intersects the interior of $\alpha$ for the first time on each side and name the intersection points $x'$ and $y'$, respectively. We note that when we extends $\beta$ past $x$ along $b$ we consider the opposite orientation of $b$. In the case where the extended b-arc does not intersect with $\alpha$ any more, we find that $i(c,b)\leq 1$ and we are done. In the case where $x'=y'$, we see that $i(c,b)\leq 2$ and we are also done. 

Therefore, we may assume that $x'\not=y'$. Let $c_{1}$ be a bicorn between $a$ and $b$ whose coners are $x, y'$, and $c_{2}$ a bicorn between $a$ and $b$ whose coners are $x', y$. We claim that for any $i=1,2$, we have $i(c,c_{i})\leq 1$. We prove this claim from now on. Without loss of generality we put $c'=c_{1}=\alpha'\cup\beta'$, where $\alpha'$ is a subarc of $a$ and $\beta'$ is a subarc of $b$, and we put $z=y'$. We note that $\alpha'\subset\alpha$ and $\beta\subset\beta'$, hence $c<c'$, and $\beta'$ intersects with $\alpha$ exactly once, that is at $z$, in its interior. We will prove the claim by examining the cases below (8 cases).
\begin{itemize}
\item the signs in $\mathrm{nbd}(\alpha)$ of the intersection points $y, z$ between $a$ and $b$ are the same or different,
\item the signs in $\mathrm{nbd}(\alpha)$ of the intersection points $x, y$ between $a$ and $b$ are the same or different, and
\item $c$ is two-sided or one-sided.
\end{itemize}

First we consider the case where the signs in $\mathrm{nbd}(\alpha)$ of the intersection points $y, z$ between $a$ and $b$ are the same, and the signs in $\mathrm{nbd}(\alpha)$ of the intersection points $x, y$ between $a$ and $b$ are the same. Then we see that $i(c,c')=1$ if $c$ is two-sided and $i(c,c')=0$ if $c$ is one-sided (Figure~\ref{fig_case1_v2}). We will proceed with this discussion in all cases. When the signs in $\mathrm{nbd}(\alpha)$ of the intersection points $y, z$ between $a$ and $b$ are the same and the signs in $\mathrm{nbd}(\alpha)$ of the intersection points $x, y$ between $a$ and $b$ are different, we see that $i(c,c')=1$ if $c$ is two-sided, and $i(c,c')=0$ if $c$ is one-sided. When the signs in $\mathrm{nbd}(\alpha)$ of the intersection points $y, z$ between $a$ and $b$ are different and the signs in $\mathrm{nbd}(\alpha)$ of the intersection points $x, y$ between $a$ and $b$ are the same, we see that $i(c,c')=0$ if $c$ is two-sided, and $i(c,c')=1$ if $c$ is one-sided. When the signs in $\mathrm{nbd}(\alpha)$ of the intersection points $y, z$ between $a$ and $b$ are different and the signs in $\mathrm{nbd}(\alpha)$ of the intersection points $x, y$ between $a$ and $b$ are different, we see that $i(c,c')=0$ if $c$ is two-sided, and $i(c,c')=1$ if $c$ is one-sided. Hence, we conclude that $i(c,c')\leq 1$. (We note that these cases can be rephrased in this order if the signs in $\mathrm{nbd}(\alpha)$ of the intersection points $x$, $y$, and $z$ are the same, if the sign in $\mathrm{nbd}(\alpha)$ of $x$ is different from those of $y$ and $z$, if the sign in $\mathrm{nbd}(\alpha)$ of $z$ is different from those of $x$ and $y$, and if the sign in $\mathrm{nbd}(\alpha)$ of $y$ is different from those of $x$ and $z$, respecvively.)

\begin{figure}[ht]
\includegraphics[scale=0.35]{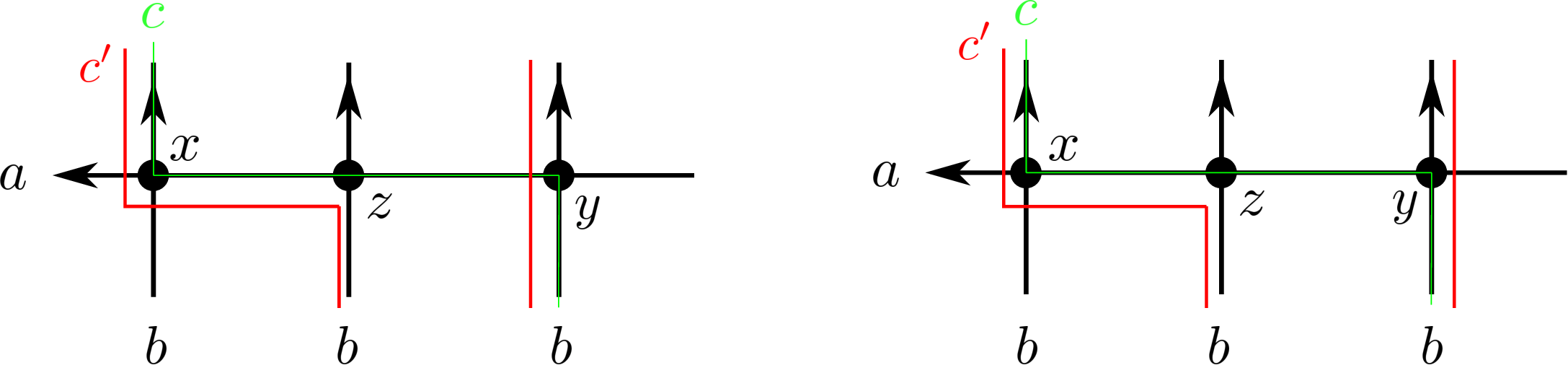}
\caption{Intersections between $c$ and $c'$ when the signs in $\mathrm{nbd}(\alpha)$ of the intersection points $y, z$ between $a$ and $b$ are the same and the signs in $\mathrm{nbd}(\alpha)$ of the intersection points $x, y$ between $a$ and $b$ are the same, and $c$ is two-sided (left-hand side) and one-sided (right-hand side).}\label{fig_case1_v2}
\end{figure}

If $c_{i}$ for some $i$ is nonseparating, we may take $c'=c_{i}$. Then $c'$ is a nonseparating curve which is adjacent to $c$ in $A(a,b)$ and satisfies $c<c'$.
%If at least one of $c_{1}$ and $c_{2}$ is nonseparating (we can choose $c_{1}$, for example), then we put $c'=c_{1}$, and we have $c<c'$ and $c$ and $c'$ are adjacent in $A(a,b)$.

Hence, we assume that both $c_{1}$ and $c_{2}$ are separating. In this case, there exists a bicorn $e_{2}$ between $a$ and $b$ such that $[c_{2}]+[e_{2}]=[c]$ when all three are appropriately oriented. Figure~\ref{fig_example_of_e2_v3} illustrates an example how to take $e_{2}$. 
Since $c_{2}$ is separating and $c$ is nonseparating, the curve $e_{2}$ is nonseparating. We define a bicorn $c'$ between $a$ and $b$ with corners $x'$ and $y'$ such that $[c']=[c_{1}]+[e_{2}]$. Then $[c']\not=0$ and $c'$ satisfies $c<c'$. We see $i(c,c')\leq 2$ by similar argument as before drawn in Figures~\ref{fig_intersection_number_between_bicorns_for_same_signs_v1}, \ref{fig_intersection_number_between_bicorns_for_opposite_signs_v1}, and \ref{fig_case1_v2}.
%It remains to prove $i(c,c')\leq 2$. Put $c'=\alpha'\cup\beta'$, where $\alpha'$ and $\beta'$ are the $a$-arc and the $b$-arc of $c'$, respectively. We see that $\alpha'\subset\alpha$ and $\beta\subset\beta'$. We can homotope $c$ and $c'$ to curves (we also write $c$ and $c'$ for the two curves, respectively) which intersect at most four points $x$, $y$, $x'$, and $y'$. If $c$ and $c'$ intersect at both $x$ and $y$, then $c$ and $c'$ form a bigon parallel to $\beta$. Then, the two intersection points $x$ and $y$ can be reduced through the bigon.  Similarly, if $c$ and $c'$ intersect at both $x'$ and $y'$, then $c$ and $c'$ form a bigon parallel to $\alpha'$, and the intersection points $x'$ and $y'$ can be reduced. Therefore, we conclude that $i(c, c')\leq 2$.
\end{proof}

\begin{figure}[ht]
\includegraphics[scale=0.35]{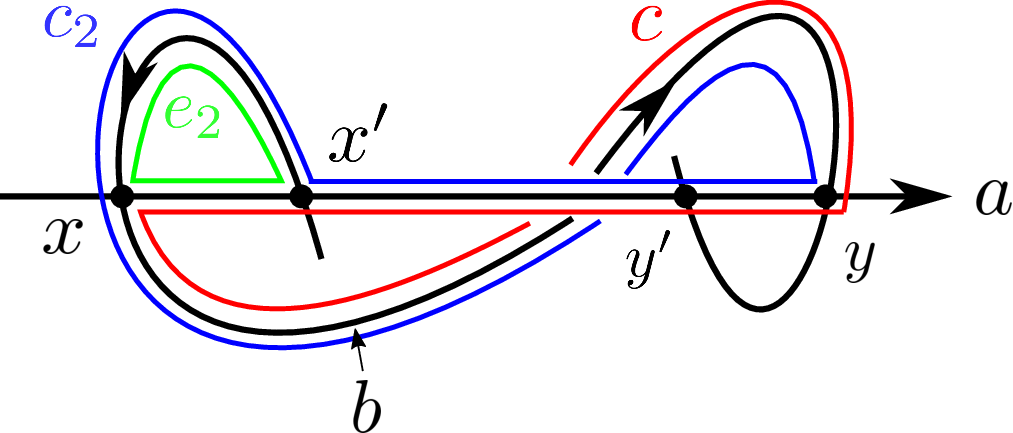}
\caption{A bicorn $e_{2}$ between $a$ and $b$.}\label{fig_example_of_e2_v3}
\end{figure}

\begin{remark}\label{difference_between_orientable_and_nonorientable_in_claim}
We note a difference from the cace of orientable surfaces (\cite{Rasmussen20}) in the proof of Lemma~\ref{connectedness_bicorn_graph}. For the orientable surfaces, if the signs of the intersection points $y, z$ between $a$ and $b$ are the same, then we can conclude that the geometric intersection number of $c$ and $c'$ is exactly $1$, that is, $i(c,c')=1$. And this implies that $c'$ is a nonseparating curve. 
%Then conditions in orientable surface case (\cite{Rasmussen20}) can be restricted to the case where the signs of the intersection points $y, z$ between $a$ and $b$ are different. 
On the other hand for nonorientable surfaces, even the case where the signs of the intersection points $y, z$ between $a$ and $b$ are the same, we have both $i(c,c')=0$ and $i(c,c')=1$, and we cannnot conclude that $c'$ is nonseparating. And so, we treat the both cases where the signs of the intersection points $y, z$ between $a$ and $b$ are the same and different at the same time in our proof.
\end{remark}

\begin{lemma}{\rm(cf. }\cite[Claim 3.5]{Rasmussen20}{\rm)}\label{slim_triangle_property}
There exists a uniform constant $D>0$ such that for any triple $a$, $b$, $d$ of vertices in $\mathcal{NC}(N)$, the graph $A(a,b)$ is contained in the $D$-neighborhood of $A(a,d)\cup A(b,d)$.
\end{lemma}

\begin{proof}[Proof of Lemma~\ref{slim_triangle_property}]
Let $c\in A(a,b)$ be any nonseparating bicorn between $a$ and $b$ and put $c=\alpha\cup\beta$, where $\alpha$ is the $a$-arc of $c$ and $\beta$ is the $b$-arc of $c$. Let $s$ and $t$ be the corners of the bicorn $c$ between $a$ and $b$. If $d$ intersects both $\alpha$ and $\beta$ at most once, we have $i(c,d)\leq 2$ and hence $c$ is at distance 1 from $d\in A(a,d)$. Otherwise, $d$ intersects at least one of $\alpha$ and $\beta$ at least twice. We assume that $d$ intersects $\beta$ at least twice in this proof. We fix an orientation of $d$, and let $x_{1}, x_{2}, \cdots, x_{m}$ ($m\geq 2$) be the intersection of $d$ and $\beta$ along $d$ in this order according to the orientation of $d$. For any $i=1,\cdots, m$, let $\delta_{i}=\overline{x_{i}x_{i+1}}_{d}$ be the unique subarc of $d$ bounded by $x_{i}$ and $x_{i+1}$ whose orientation agrees with the orientation of $d$ (indices are taken modulo $m$). Also, let $\beta_{i}$ be the unique subarc of $\beta$ such that $c_{i}'=\beta_{i}\cup\delta_{i}$ is a bicorn between $b$ and $d$ (we orient $\beta$ from $x_{i+1}$ to $x_{i}$). If we give $c_{i}'$ the orientation induced by $\delta_{i}$, we have
\begin{equation*}
[d]=[c_{1}']+[c_{2}']+\cdots+[c_{m}'].
\end{equation*}
Since $[d]\not=0$, there is some $u$ such that $[c_{u}']\not=0$. Put $c'=c_{u}'$. We note that $c'\in A(b,d)$. Let $\beta'$ be the $b$-arc of $c'$. If $c'$ intersect with $\alpha$ at most once, then it follows that $i(c,c')\leq 2$, and $c$ is at distance 1 from $c'\in A(b,d)$ in $\mathcal{NC}(N)$.

If $c'$ intersects with $\alpha$ at least twice, we denote by $y_{1}, y_{2}, \cdots, y_{n}$ ($n\geq 2$) the points of $c'\cap\alpha$ appearing in this order along $c'$ (i.e. $d$) according to the orientation of $d$ from one of the endpoints of $\beta'$ (we refer \cite[Figure 7]{Rasmussen20} as a figure which describes this situation).
%(without loss of generality, we may take the endpoint $x_{u}$). 
For any $i=1,\cdots,n-1$, the points $y_{i}$ and $y_{i+1}$ bound a unique arc $\alpha_{i}$ of $\alpha$ and a unique arc $\gamma_{i}$ of $c'$ not containing the $b$-arc $\beta'$. Then, $c_{i}'=\alpha_{i}\cup\gamma_{i}$ is a bicorn between $a$ and $d$. It follows that $i(c,c_{i}'')\leq1$. Actually, the interior of $\gamma_{i}$ does not intersect with $c$, and hence, the intersection occurs only on $\alpha_{i}$. %If the signs of $c\cap c_{i}''$ at $y_{i}$ and $y_{i+1}$ are the same, then $i(c,c_{i}'')=0$ when $\alpha_{i}$ is twisted and $i(c,c_{i}'')=1$ when $\alpha_{i}$ is untwisted, and if the signs of $c\cap c_{i}''$ at $y_{i}$ and $y_{i+1}$ are different, then $i(c,c_{i}'')=1$ when $\alpha_{i}$ is twisted and $i(c,c_{i}'')=0$ when $\alpha_{i}$ is untwisted. 
Therefore, if there exists $c_{i}''$ ($i=1,2,\cdots,n-1$) which is nonseparating, then $c''=c_{i}''$ is an element of $A(a,d)$ and $d_{\mathcal{NC}}(c, c'')=1$.

Thus let us assume that all $c_{i}''$ are separating. In this case, we claim that $c$ is a uniformly bounded distance from $c'\in A(b,d)$. To show this, we orient the regular neighborhood of $\alpha$ (note that the regular neighborhood of $\alpha$ is homeomorphic to a disk.) and make the following two observations:
\begin{itemize}
\item[(I)] In the oriented regular neighborhood of $\alpha$, for each $i=1,\cdots, n-1$, $d$-arc $\gamma_{i}$ joins the left-hand side (resp.\ right-hand side) of $\alpha$ to the left-hand side (resp.\ right-hand side) of $\alpha$ following the orientation of the regular neighborhood. Otherwise, $c_{i}''$ intersects with $c$ exactly once, and it indicates that $c_{i}''$ is nonseparating.
%In fact, if $\gamma_{i}$ does not connect the same side of $\alpha$ with respect to the orientation of the regular neighborhood of $\alpha$, then $\gamma_{i}$ intersects $\alpha$ once, so we have $i(c_{i}'', c)=1$, and it contradicts the assumption that $c_{i}''$ is separating,
\item[(I\hspace{-1pt}I)] For each $1\leq i<j\leq n-1$, if in the oriented regular neighborhood of $\alpha$ the $d$-arcs of $c_{i}''$ and $c_{j}''$, namely, $\gamma_{i}$ and $\gamma_{j}$ join the left-hand side of $\alpha$ to the left-hand side of $\alpha$ following the orientation of the regular neighborhood of $\alpha$, then the $a$-arcs of $c_{i}''$ and $c_{j}''$, namely, $\alpha_{i}$ and $\alpha_{j}$, are either nested or disjoint. In fact, we suppose that $\gamma_{i}$ has the endpoints $x$, $z$ and $\gamma_{j}$ has the endpoints $y$, $w$ with $x<y<z<w$ in the orientation of $\alpha$. Then the bicorns $c_{i}''$ and $c_{j}''$ have $i(c_{i}'',c_{j}'')=1$. This contradicts that $c_{i}''$ and $c_{j}''$ are separating. 
%For suppose that $\gamma_{i}$ has the endpoints $x$, $z$ and $\gamma_{j}$ has the endpoints $y$, $w$ with $x<y<z$. Note that $c_{i}''$ separats $N$ into two components. If the arc $\gamma_{j}$ comes in the one component by intersecting $\alpha_{i}$ (we consider the intersection point as $y$), then $\gamma_{j}$ should get  out the component by intersecting $\alpha_{i}$, and the intersection point is $w$ between $x<z$. Therfore we have $\alpha_{j}\subset\alpha_{i}$. Otherwise, we have $\alpha_{i}$ and $\alpha_{j}$ are disjoint.
\end{itemize}

Now enumerate all of the $d$-arcs $\gamma_{j}$ which join the left-hand side of $\alpha$ to the left-hand side of $\alpha$ in the oriented regular neighborhood of $\alpha$ such that the corresponding $a$-arcs $\alpha_{j}$ are maximal with respect to the inclusion: $\gamma_{1}',\cdots,\gamma_{r}'$. Note that $\gamma_{i}'\in\{\gamma_{1},\cdots,\gamma_{n-1}\}$ for $i=1,\cdots, r$. Let $\alpha_{1}',\cdots,\alpha_{r}'$ be the corresponding $a$-arcs. For each $i=1,\cdots,r$, we put $c_{i}'''=\alpha_{i}'\cup\gamma_{i}'$. Note that $c_{i}'''\in\{c_{1}'', c_{2}'',\cdots,c_{n-1}''\}$ for $i=1,\cdots, r$, and $c_{i}'''$ is a bicorn between $a$ and $d$. For each $i=1,\cdots,r$, we replace the subarc $\alpha_{i}'$ of $c$ with $\gamma_{i}'$ and call the resulting curve $c_{0}$. We see that $[c_{0}]=[c]+[c_{1}''']+[c_{2}''']+\cdots+[c_{r}''']$ in $H_{1}(N,\partial N; \mathbb{Z}_{2})$. Since $[c_{i}''']=0$ for all $i=1,\cdots, r$, we have $[c_{0}]=[c]$. Thus, $c_{0}$ is nonseparating, that is, $c_{0}\in\mathcal{NC}(N)$. 

Finally, we claim that $i(c,c_{0})\leq 1$ and $i(c_{0}, c')\leq 3$. 
First, we show the first half. Recall that $c=\alpha\cup\beta$. Note that $\beta$ is properly contained in $c_{0}$. Then, by similar argument as before drawn in Figures~\ref{fig_intersection_number_between_bicorns_for_same_signs_v1}, \ref{fig_intersection_number_between_bicorns_for_opposite_signs_v1}, and \ref{fig_case1_v2}, we can show $i(c,c_{0})\leq 1$ (The intersection occurs only on $\beta$). Figure~\ref{fig_the_curve_c0_v2} is a possible intersection pattern of $c$ and $c_{0}$. (Note that in the case of orientable surfaces~\cite[Claim 3.5]{Rasmussen20}, the intersection number is exactly $i(c,c_{0})=0$.)
%We can homotope $c$ and $c_{0}$ to curves (we also call them $c$ and $c_{0}$, respectively) which intersect only at the corners $s$ and $t$ of $c$. If $c$ and $c_{0}$ intersect at both $s$ and $t$, then they forms a bigon parallel to $\beta$ since $c_{0}$ contains $\beta$. We can reduce the intersection points $s$ and $t$ through the bigon. Hence we see $i(c, c_{0})\leq 1$ (Figure~\ref{fig_the_curve_c0_v2} is a possible intersection pattern of $c$ and $c_{0}$).

Now we show the second half. Recall that $c_{0}$ consists of $\beta$ and several $\alpha$-arcs and $d$-arcs. In particular, $\beta\subset c_{0}$. Moreover, $c'$ is a bicorn between $b$ and $d$ whose corners are $x_{u}$ and $x_{u+1}$. In particular, $x_{u},x_{u+1}\in\beta$. Then we see that the number of intersection points between $c_{0}$ and $c'$ on $\beta$ is at most one by similar argument as before drawn in Figures~\ref{fig_intersection_number_between_bicorns_for_same_signs_v1}, \ref{fig_intersection_number_between_bicorns_for_opposite_signs_v1}, and \ref{fig_case1_v2}. 
Any other intersection points of $c_{0}$ with $c'$ ocuur between $a$-arcs of $c_{0}$ and the $d$-arc of $c'$. By the definition of $c_{0}$, we see that such intersection points occur only on $y_{1}$ and $y_{n}$. This concludes that $i(c_{0},c')\leq 3$. Figure~\ref{fig_intersection_of_c0_and_c_prime_v3} is a possible intersection pattern of $c_{0}$ and $c'$.
%We can homotope $c_{0}$ and $c'$ to curves (we also call them $c_{0}$ and $c'$, respectively) which intersect at most four points: $x_{u}$, $x_{u+1}$, $y_{1}, y_{n}$. If $c_{0}$ and $c'$ intersect at both $x_{u}$ and $x_{u+1}$, then $c_{0}$ and $c'$ form a bigon along $\beta'$, the both intersection points $x_{u}$ and $x_{u+1}$ can be reduced throgh the bigon. Hence, $i(c_{0},c')\leq 3$. Figure~\ref{fig_intersection_of_c0_and_c_prime_v3} is a possible intersection pattern of $c_{0}$ and $c'$.

\end{proof}

\begin{figure}[ht]
\includegraphics[scale=0.35]{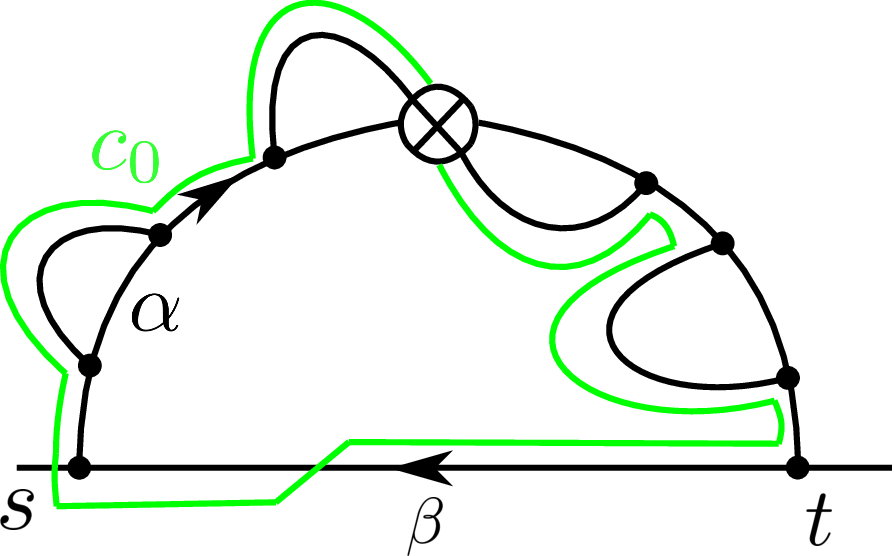}
\caption{A curve $c_{0}$ which intersects with $c$ once, where the antipodal points of a crosscap are identified.}\label{fig_the_curve_c0_v2}
\end{figure}

\begin{figure}[ht]
\includegraphics[scale=0.35]{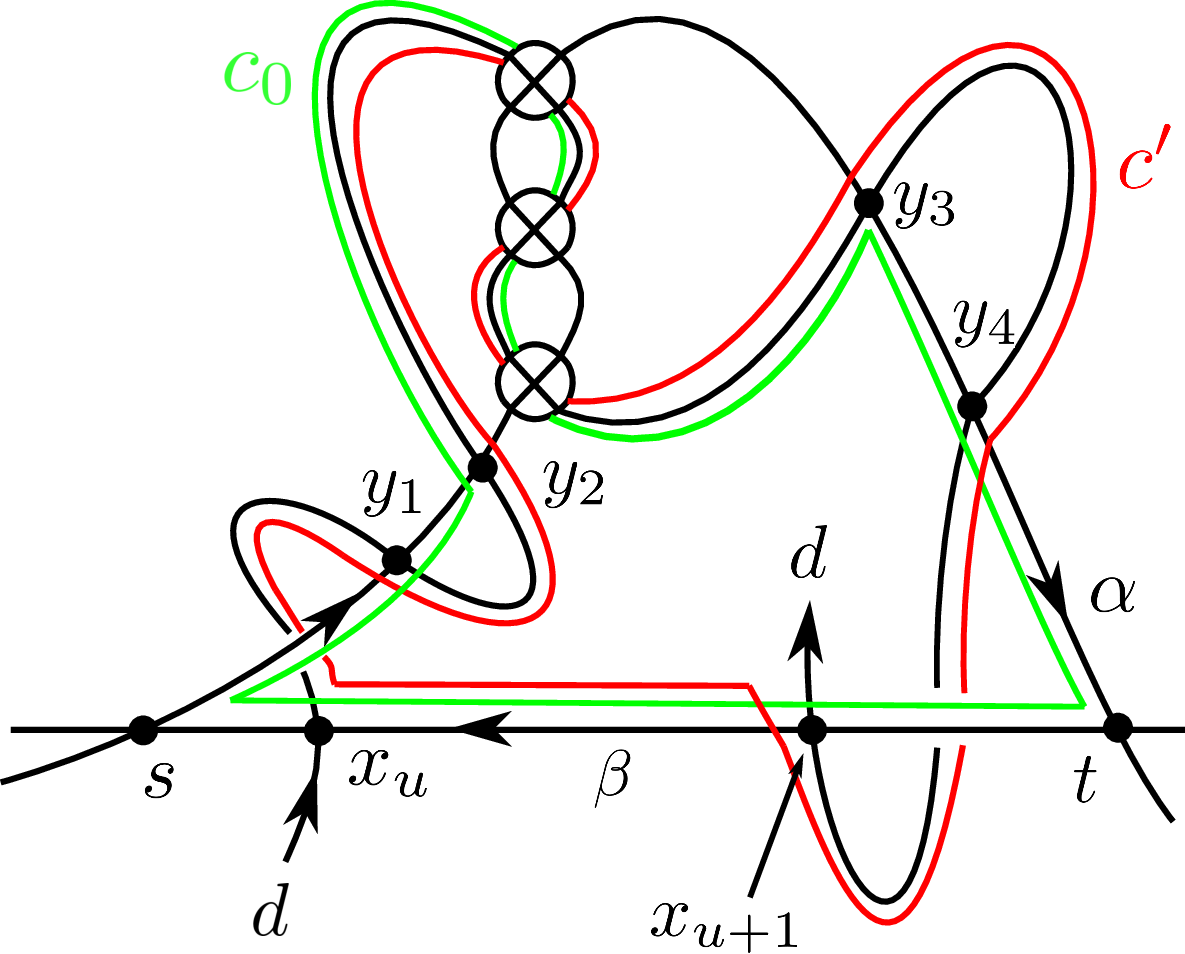}
\caption{Curves $c_{0}$ and $c'$ intersecting three times.}\label{fig_intersection_of_c0_and_c_prime_v3}
\end{figure}

\begin{remark}
We can take the uniform constant $D=10$ in Lemma~\ref{slim_triangle_property}. Indeed, let $c$, $c'$, and $c_{0}$ be the nonseparating curves in the proof of Lemma~\ref{slim_triangle_property}. At the last of the proof of Lemma~\ref{slim_triangle_property}, we showed that $i(c,c_{0})\leq 1$ and $i(c_{0}, c')\leq 3$. By Lemma~\ref{nonseparating_curve_graphs_are_quasi-isometric}, it follows that
\begin{align*}
  \begin{aligned}
  d_{\mathcal{NC}'}(c,c') & \leq
    d_{\mathcal{NC}'}(c,c_{0})+d_{\mathcal{NC}'}(c_{0},c') \\
   & \leq  2i(c,c_{0})+1+2i(c_{0},c')+1 \\
   & \leq 2+1+6+1 \\
   & \leq 10. \\
%   & \leq 2i(a,b)+1.
\end{aligned}
\end{align*}
Thus, we see that it is enough to take 10 as $D$ for the constant in Lemma~\ref{slim_triangle_property}.
\end{remark}

\section{The cases of genus 1 and 2}\label{The_cases_of_genus_1 and_2}
%\section{The case of genus 2 and one boundary component}\label{The_cases_of_genus_1 and_2}
%\section{The case of $g=2$ and $f=1$}\label{The_cases_of_genus_1 and_2}
%\section{The case of $N=N_{2}^{1}$}\label{The_cases_of_genus_1 and_2}

%In this section, we denote by $\mathcal{C}_{u}(F)$ and $\mathcal{NC}_{u}(F)$ the usual curve graph and nonseparating curve graphwhose edges are the pairs of vertices which can be realized disjointly.

As mentioned in Remark~\ref{disconnectedness_nonseparating_curve_graphs_for_low_genus}, for nonorientable surfaces of genus 1 and 2, the nonseparating curve graphs $\mathcal{NC}'(N)$ are finite or not connected. In the cases of genus 1 and 2, we modify the definition of $\mathcal{NC}'(N)$ so that two vertices are joined by an edge if they are represented by curves which intersect at most once. Then we see that $\mathcal{NC}'(N)$ are connected (Proposition~\ref{connectedness_low_genus}). % by using the argument in the proof of Lemma~\ref{nonseparating_curve_graphs_are_quasi-isometric}.
We also denote by $\mathcal{C}'(F)$ the curve graph whose vertises are the homotopy classes of essential simple closed curves, and two vertices are joined by an edge if they are represented by curves which intersect at most once.
In this section, we denote by $\mathcal{C}_{u}(F)$ and $\mathcal{NC}_{u}(F)$ the usual curve graph and nonseparating curve graph, respectively, that is, the vertices are homotopy classes of essential simple closed curves and nonseparating curves, respectively, and the edges are the pairs of vertices which can be realized disjointly. We also denote by $d_{\mathcal{C}'(F)}(\cdot,\cdot)$, $d_{\mathcal{NC}'(F)}(\cdot,\cdot)$, $d_{\mathcal{C}_{u}(F)}(\cdot,\cdot)$, and $d_{\mathcal{NC}_{u}(F)}(\cdot,\cdot)$ the distances of $\mathcal{C}'(F)$, $\mathcal{NC}'(F)$, $\mathcal{C}_{u}(F)$, and $\mathcal{NC}_{u}(F)$, respectively.

Firstly, we prove the connectedness of the nonseparating curve graphs $\mathcal{NC}'(N)$ of genus 1 and 2: 
\begin{proposition}\label{connectedness_low_genus}
Let $N=N_{g}^{f}$ be a nonorientable surface of $g=1$ and $f\geq 3$ or $g=2$ and $f\geq 1$. 
Then,  $\mathcal{NC}'(N)$ is connected.
\end{proposition}
\begin{proof}
At the first we consider nonorientable surfaces of genus $g=1$. For $f=0,1$, the nonseparating curve graph $\mathcal{NC}'(N)$ consists of exactly one isolated vertex.
For $f\geq 2$, each nonseparating curve on $N$ goes through the unique crosscap of $N$, and so for any pair of vertices $a,b\in\mathcal{NC}'(N)$ the geometric intersection number $i(a,b)$ is larger than 0. We prove by the induction on  $i(a,b)$.
When $i(a,b)=1$, then $a$ and $b$ are connected by an edge of $\mathcal{NC}'(N)$.
When $i(a,b)\geq 2$, we choose an orientation of $b$. We take an intersection point $x\in a\cap b$. Let $y$ be the first point of $a\cap b$ after $x$ along $b$ according to the orientation of $b$. Then, we can construct two curves $c_{1}$ and $c_{2}$ as shown in Figure~\ref{fig_how_to_make_c_1_and_c_2_v3} as the proof of Lemma~\ref{nonseparating_curve_graphs_are_quasi-isometric}. By the same argument as the proof of Lemma~\ref{nonseparating_curve_graphs_are_quasi-isometric}, we see that at least one of $c_{1}$ and $c_{2}$ is nonseparating, and denote it by $c$. Then it follows that $i(a,c)=1$ and $i(b,c)\leq i(a,b)-1$, and we can connect $a$ and $c$ by an edge and connect $c$ and $b$ by a path in $\mathcal{NC}'(N)$ by induction. Hence $a$ and $b$ are connected by a path in $\mathcal{NC}'(N)$.

Second, we consider nonorientable surfaces of genus $g=2$. For $f=0,1$, the nonseparating curve graph $\mathcal{NC}'(N)$ is connected by \cite[Section 2.4]{Atalan--Korkmaz14} and our definition of $\mathcal{NC}'(N)$. 
%For $f=0$, the nonseparating curve graph $\mathcal{NC}'(N)$ consists of three vertices which are mutually connected by an edge (for the three vertices see \cite[Section 2.4]{Atalan--Korkmaz14}). For $f=1$, the nonseparating curve graph $\mathcal{NC}'(N)$ consists of one vertex which is represented by the unique two-sided curve and countably infinite vertices whch are represented by one-sided curves.
For $f\geq 2$, we can prove the nonseparating curve graph $\mathcal{NC}'(N)$ is connected by the same argument as the case of $g=1$ and $f\geq 2$ above.
\end{proof}

%We also see that for $N=N_{2}^{1}$, this augmented nonseparating curve graph $\mathcal{NC}'(N)$ has infinite diameter  (see \cite[Section 2.4]{Atalan--Korkmaz14}).
%In the case where a nonseparating curve graph is infinite but disconnected, that is $g=1$ and $p\geq 2$ or $g=2$ and $p\geq 1$, we modify the definition of $\mathcal{NC}'(N)$ so that two vertices are joined by an edge if they represent curves with intersection number at most once.

If $g=1$ and $f\leq 2$ or $g=2$ and $f=0$, $\mathcal{NC}'(N_{g}^{f})$ is a finite graph, and hence it is Gromov hyperbolic. If $g=1$ and $f\geq 3$ or $g=2$ and $f\geq 1$, $\mathcal{NC}'(N_{g}^{f})$ has infinite vertices and we see the following:

\begin{proposition}\label{infinite_diameter_low_genus}
Let $N=N_{g}^{f}$ be a nonorientable surface of $g=1$ or $g=2$. Then, $\mathcal{NC}'(N)$ has infinite diameter.
\end{proposition}

\begin{proof}
Let $S$ be the orientation double cover of $N$.
We note that the vertex sets of $\mathcal{C}_{u}(S)$ and $\mathcal{C}'(S)$ are the same.
Firstly we show that any pseudo-Anosov element $\varphi$ of $\mathrm{Mod}(S)$ acts on $\mathcal{C}'(S)$ loxodromically.
%By the natural inclusion $\iota'\colon\mathcal{C}_{u}(S)\hookrightarrow\mathcal{C}'(S)$, we see that for any vertices $a$ and $b$ of $\mathcal{C}_{u}(S)$, $d_{\mathcal{C}'(S)}(a,b)\leq d_{\mathcal{C}_{u}(S)}(a,b)$. 
By Bowditch~\cite[Lemma 1.1]{Bowditch06} (also Schleimer~\cite[Lemma 1.21]{Schleimer06}), we see that for any vertices $a$ and $b$ of $\mathcal{C}_{u}(S)$, $d_{\mathcal{C}_{u}(S)}(a,b)\leq\log_{2}(i(a,b))+2$. %Therefore we see that for $a,b\in\mathcal{C}_{u}(S)^{0}$($=\mathcal{C}'(S)^{0}$), $d_{\mathcal{C}_{u}(S)}(a,b)\leq 2d_{\mathcal{C}'(S)}(a,b)$. 
Then since two vertices $a$ and $b$ of $\mathcal{C}'(S)$ are connected by an edge if $i(a,b)\leq 1$, we see that for any vertices $a$ and $b$ of $\mathcal{C}_{u}(S)$, $d_{\mathcal{C}_{u}(S)}(a,b)\leq 2d_{\mathcal{C}'(S)}(a,b)$. 
%Therefore we see that $\iota'$ is quasi-isometric embedding.
Let $\varphi$ be any pseudo-Anosov element of $\mathrm{Mod}(S)$. By Masur and Minsky~\cite{Masur--Minsky99}, $\varphi$ acts on $\mathcal{C}_{u}(S)$ loxodromically. Hence, for any vertex $c$ of $\mathcal{C}'(S)$, there exists a constant $\mathcal{E}>0$ such that for any $n\in\mathbb{Z}$, we have $\mathcal{E}|n|\leq d_{\mathcal{C}_{u}(S)}(c,\varphi^{n}(c))\leq d_{\mathcal{C}'(S)}(c,\varphi^{n}(c))$. Therefore any pseudo-Anosov element $\varphi$ of $\mathrm{Mod}(S)$ acts on $\mathcal{C}'(S)$ loxodromically.

From now on let $\varphi$ be any pseudo-Anosov element of $\mathrm{Mod}(N)$ and $c$ a vertex of $\mathcal{C}'(N)$. We fix any $n\in\mathbb{Z}$. By applying the same argument in the proof of Proposition~\ref{infinite_diameter} to $\mathcal{C}'(N)$ and $\mathcal{C}'(S)$, it follows that there exists a constant $\mathcal{E}>0$ such that for any $n\in\mathbb{Z}$, we have $\mathcal{E}|n|\leq d_{\mathcal{C}'(N)}(c,\varphi^{n}(c))$. We remark that if $\delta_{i}$ and $\delta_{i+1}$ are vertices of $\mathcal{C}'(N)$ which are connected by an edge in $\mathcal{C}'(N)$, then we have $i(\delta_{i},\delta_{i+1})\leq 1$. Let $\gamma^{i}$ and $\gamma^{i+1}$ be lifts of $\delta_{i}$ and $\delta_{i+1}$ to $S$, respectively. If $i(\delta_{i},\delta_{i+1})=0$, then $i(\gamma^{i},\gamma^{i+1})=0$, and if $i(\delta_{i},\delta_{i+1})=1$, then $i(\gamma^{i},\gamma^{i+1})\leq 1$. Hence $\gamma^{i}$ and $\gamma^{i+1}$ are also connected by an edge in $\mathcal{C}'(S)$, and we can use the same argument in the proof of Proposition~\ref{infinite_diameter}.
Since the nonseparating curve graph $\mathcal{NC}'(N)$ is a full subgraph of the curve graph $\mathcal{C}'(N)$, for any $c\in\mathcal{NC}(N)$, there exists $\mathcal{E}>0$ such that for any $n\in\mathbb{Z}$,
\begin{equation*}
\mathcal{E}|n|\leq d_{\mathcal{C}'(N)}(c,\varphi^{n}(c))\leq d_{\mathcal{NC}'(N)}(c,\varphi^{n}(c)),
\end{equation*}
and we see that the nonseparating curve graph $\mathcal{NC}'(N)$ has infinite diameter.
%Since there exists a natural inclusion $\mathcal{NC}'(N)\hookrightarrow\mathcal{C}'(N)$, it follows that for any vertex $c$ of $\mathcal{NC}'(N)$, there exists a constant $\mathcal{E}>0$ such that for any $n\in\mathbb{Z}$, we have $\mathcal{E}|n|\leq d_{\mathcal{NC}'(N)}(c,\varphi^{n}(c))$. 

%By a similar argument of the proof of Proposition~\ref{infinite_diameter}, we can show that $\mathcal{NC}'(N)$ has infinite diameter.

\end{proof}

Moreover, we obtain a modification of Lemma~\ref{nonseparating_curve_graphs_are_quasi-isometric} as below:

\begin{lemma}\label{augmented_nonseparating_curve_graphs_are_quasi-isometric} 
Let $N=N_{g}^{f}$ be a nonorientable surface of $g=1$ and $f\geq 3$ or $g=2$ and $f\geq 1$. 
Let $a$ and $b$ be any pair of vertices of $\mathcal{NC}'(N)$. Then, we have $d_{\mathcal{NC}'}(a,b)\leq 2i(a,b)+1$.
\end{lemma}

\begin{proof}[Proof of Lemma~\ref{augmented_nonseparating_curve_graphs_are_quasi-isometric}]
We prove the lemma by induction on $i(a,b)$.
In the base case where $i(a,b)\leq 1$, we have $d_{\mathcal{NC}'}(a,b)\leq 1$ by the definition.
The rest of the proof goes through unmodified (see Proof of Lemma~\ref{nonseparating_curve_graphs_are_quasi-isometric}).
\end{proof}

By Lemma~\ref{augmented_nonseparating_curve_graphs_are_quasi-isometric}, we can see $\mathcal{NC}'(N)$ and $\mathcal{NC}(N)$ are quasi-isometric, and so we see the hyperbolicity of the augmented nonseparating curve graph for $N=N_{g}^{f}$ of $g=1$ and $f\geq 3$ or $g=2$ and $f\geq 1$ by the same argument as we did in Section~\ref{Nonseparating curve graphs of nonorientable surfaces are uniformly hyperbolic}.
%We note that the diameters of the augmented nonseparating curve graphs of nonorientable surfaces $N=N_{g}^{f}$ of $g=1$ and $f=0,1$ or $g=2$ and $f=0$ are finite, and $g=2$ and $f=1$ is infinite. However we do not know whether the diameters of the augmented nonseparating curve graphs of nonorientable surfaces $N=N_{g}^{f}$ of $g=1$ and $f\geq 2$ or $g=2$ and $f\geq 1$ are infinite.

%\begin{remark}
%We remark on the other nonorientable surfaces of genus 1 and 2. For $g=1$ and $f\leq 2$ or $g=2$ and $f=0$ we see that $\mathcal{NC}'(N)$ are finite (see Atalan and Korkmaz~\cite{Atalan--Korkmaz14}). For $g\geq 1$ and $f\geq 2$ or $g\geq 2$ and $f\geq 2$, $\mathcal{NC}'(N)$ have infinite vertices, and  we see the uniform hyperbolicity of the augmented nonseparating curve graphs of these nonorientable surfaces by the same argument as we did above. However the author does not know whether they have infinite diameters.
%%For the other nonorientable surfaces of genus 1 and 2, $\mathcal{NC}'(N)$ are finite for $g=1$ and $f=0,1$ or $g=2$ and $f=0$, and are infinite for $g\geq 1$ and $f\geq 2$ or $g\geq 2$ and $f\geq 2$. 
%\end{remark}

\vspace{1.0mm}
\par
{\bf Acknowledgements:} The author wishes to thank Mitsuaki Kimura to giving her the encouragement to work on this research and for valuable advice. The author is also deeply grateful to B\l a\.{z}ej Szepietowski, Genki Omori, and Takuya Katayama for answering her questions very kindly. The author was supported by the Foundation of Kinoshita Memorial Enterprise, by a JSPS KAKENHI Grant-in-Aid for Early-Career Scientists, Grant Number 21K13791, and by JST, ACT-X, grant number JPMJAX200D.

\end{document}